\documentclass{article}
\usepackage[utf8]{inputenc}

\usepackage{amsmath,amssymb,amsthm}
\usepackage{bm}
\usepackage{algorithm,algorithmic}
\usepackage{enumerate}
\usepackage[margin=25mm]{geometry}
\usepackage[numbers]{natbib}
\usepackage{mathtools}
\mathtoolsset{showonlyrefs}

\newcommand{\argmin}{\mathop{\rm argmin}\limits}
\newcommand{\dom}{\mathop{\rm dom}}

\newcommand{\innerprod}[2]{\langle #1,#2 \rangle}

\newtheorem{theorem}{Theorem}[section]
\newtheorem{corollary}{Corollary}[section]
\newtheorem{lemma}{Lemma}[section]

\newtheorem{proposition}{Proposition}[section]

\title{Convergence of linesearch-based generalized conditional gradient methods without smoothness assumptions}
\author{Shotaro Yagishita\thanks{Risk Analysis Research Center, The Institute of Statistical Mathematics, Japan, E-mail: syagi@ism.ac.jp} \thanks{Center for Social Data Structuring, Joint Support-Center for Data Science Research, Japan}}
\date{\today}

\begin{document}

\maketitle

\begin{abstract}
The generalized conditional gradient method is a popular algorithm for solving composite problems whose objective function is the sum of a smooth function and a nonsmooth convex function.
Many convergence analyses of the algorithm rely on smoothness assumptions, such as the Lipschitz continuity of the gradient of the smooth part.
This paper provides convergence results of linesearch-based generalized conditional gradient methods without smoothness assumptions.
In particular, we show that a parameter-free variant, which automatically adapts to the H\"older exponent, guarantees convergence even when the gradient of the smooth part of the objective is not H\"older continuous.
\end{abstract}

\section{Introduction}\label{sec:intro}
This paper considers solving composite problems whose objective function is the sum of a smooth function $f$ and a nonsmooth convex function $g$.
As many problems in machine learning, signal processing, and statistical inference can be formulated as them, such problems have attracted interest.

The generalized conditional gradient method was proposed for solving such problems by \citet{mine1981minimization}.
The generalized conditional gradient method minimizes the subproblem that is the sum of the linearization of $f$ and the original $g$, at each iteration.
The algorithm originates from the Frank--Wolfe method \citep{frank1956algorithm}, which is initially an algorithm for convex quadratic programming.
The Frank--Wolfe method was generalized to possibly nonconvex problems with compact convex constraint as the conditional gradient method by \citet{levitin1966constrained}.

Another algorithm for the composite problems is the proximal gradient method \citep{fukushima1981generalized}.
The proximal gradient method not only linearizes $f$ but also adds a quadratic proximal term to the subproblem.
When $g$ is the indicator function of a closed convex set, the proximal gradient method reduces to the projected gradient method.
In large-scale problems, linearized subproblems are sometimes more computationally efficient than quadratic subproblems such as projections \citep{clarkson2010coresets,jaggi2013revisiting,combettes2021complexity}.
Consequently, the generalized conditional gradient method have recently regained attention.
Refer also to the comprehensive survey by \citet{braun2022conditional}.

In many cases, the convergence of the generalized conditional gradient method is established under smoothness assumptions, such as the Lipschitz continuity of the gradient \citep{levitin1966constrained,dunn1978conditional,dunn1979rates,clarkson2010coresets,jaggi2013revisiting,lacoste2016convergence,beck2017first,nesterov2018complexity,ghadimi2019conditional,jiang2019structured,kunisch2024fast,upadhayay2024nonmonotone,takahashi2025accelerated}.
\citet{mine1981minimization} and \citet{bredies2009generalized} provided convergence analyses of the generalized conditional gradient method that does not require smoothness assumptions, but they rely on the exact linesearch technique, which is impractical.
Although a subsequential convergence result that does not rely on smoothness assumptions when using an Armijo-type linesearch was conducted by \citet{bertsekas1999nonlinear}, it assumes that $g$ is the indicator function of a compact convex set.
On the other hand, a recent analysis of the proximal gradient method without assuming smoothness has been conducted by \citet{kanzow2022convergence}, which has led to numerous follow-up studies \citep{de2022proximal,de2023proximal,jia2023convergence,kanzow2024convergence,yagishita2025proximal}.

Inspired by these works, this paper provides a subsequential convergence result for the generalized conditional gradient method exploiting the average-type nonmonotone linesearch \citep{zhang2004nonmonotone} without smoothness assumptions (Theorem \ref{thm:subsequential-convergence}).
The algorithm we analyze coincides with that of \citet{kunisch2024fast} in the monotone case and with that of \citet{upadhayay2024nonmonotone} when $g$ is a compact convex set.
Unlike our results, both of them impose smoothness assumptions on the gradient.
We also provide such a convergence analysis for the parameter-free variant proposed by \citet{ito2023parameter} (Theorem \ref{thm:subsequential-convergence-p-f}).
\citet{ito2023parameter} have shown that under the H\"older continuity of the gradient, their algorithm automatically achieves an appropriate convergence rate.
Our results demonstrate that even when the gradient does not have the H\"older continuity, their algorithm still possesses convergence properties.

The rest of this paper is organized as follows.
The next section is devoted to notation and preliminary results.
We provide main results of the paper in the section \ref{sec:convergence}.
Finally, Section \ref{sec:conclusion} concludes the paper with some remarks.

\section{Notation and Preliminaries}\label{sec:notation}
Let $\mathbb{E}$ be a finite-dimensional inner product space endowed with an inner product $\innerprod{\cdot}{\cdot}$.
The induced norm is denoted by $\|\cdot\|$.
Let $\phi:\mathbb{E}\to(-\infty,\infty]$ be a function.
The domain of $\phi$ is denoted by $\dom \phi\coloneqq\{x\in\mathbb{E}\mid\phi(x)<\infty\}$.
A function $\phi$ is said to be supercoercive if $\lim_{\|x\|\to\infty}\phi(x)/\|x\|=\infty$.

We consider the following composite problem:
\begin{equation}\label{problem:general}
    \underset{x\in\mathbb{E}}{\mbox{minimize}} \quad F(x)\coloneqq f(x)+g(x),
\end{equation}
where $f$ is lower semicontinuous and continuous differentiable on $\dom g$, $g$ is a proper closed convex function, and $F$ is bounded from below.
Any local minimizer $x^*\in\dom g$ of \eqref{problem:general} satisfies
\begin{equation}\label{eq:opt-condition}
    -\nabla f(x^*)\in\partial g(x^*),
\end{equation}
where $\partial g$ is the subdifferential of the convex function $g$.
We call a point $x^*\in\dom g$ satisfying \eqref{eq:opt-condition} a stationary point of \eqref{problem:general}.
When $f$ is convex, the condition \eqref{eq:opt-condition} is sufficient for the optimality.
Suppose also that $g$ is supercoercive.
In the literature \citep{bredies2009generalized,beck2017first,nesterov2018complexity,ghadimi2019conditional,jiang2019structured,ito2023parameter,kunisch2024fast}, the supercoerciveness or a stronger condition is assumed for $g$.
Then, the solution set of the subproblem of the generalized conditional gradient method:
\begin{equation}\label{eq:subproblem}
    \argmin_{v\in\mathbb{E}}\left\{\innerprod{\nabla f(x)}{v}+g(v)\right\}
\end{equation}
is nonempty for any $x\in\dom g$.
The Frank-Wolfe gap at $x\in\dom g$ is defined by
\begin{equation}
    G(x)\coloneqq\max_{v\in\mathbb{E}}\left\{\innerprod{\nabla f(x)}{x-v}+g(x)-g(v)\right\}\ge0.
\end{equation}
It is easy to see that $G(x)=\innerprod{\nabla f(x)}{x-v^*}+g(x)-g(v^*)$ for any solution $v^*$ of the subproblem \eqref{eq:subproblem}.
Therefore, the Frank-Wolfe gap is a computable quantity within the algorithm.
As can be seen from the following lemma, the Frank-Wolfe gap is an optimality measure.

\begin{lemma}[$\mbox{\citep[Theorem 13.6 \& Lemma 13.12]{beck2017first}}$]\label{lem:stationarity-F-W-gap}
$G(x)=0$ if and only if $x$ is a stationary point of \eqref{problem:general}.
Moreover, if $f$ is convex, it holds that $F(x)-F(x^*)\le G(x)$ for any $x\in\dom g$, where $x^*$ is an optimal solution of \eqref{problem:general} (it exists due to the convexity of $f$ and the supercoerciveness of $g$).
\end{lemma}

The following lemma is used in the analysis of this paper.

\begin{lemma}\label{lem:lsc-F-W-gap}
$G$ is lower semicontinuous.
\end{lemma}

\begin{proof}
Let $\{x^k\}$ be a sequence converging to $x^*$.
For all $v\in\dom g$, it holds that
\begin{equation}
    G(x^k)\ge\innerprod{\nabla f(x^k)}{x^k-v}+g(x^k)-g(v),
\end{equation}
and hence taking the lower limit yields
\begin{equation}
    \liminf_{k\to\infty}G(x^k)\ge\innerprod{\nabla f(x^*)}{x^*-v}+g(x^*)-g(v).
\end{equation}
By taking the maximum over $\dom g$, we have
\begin{equation}
    \liminf_{k\to\infty}G(x^k)\ge G(x^*),
\end{equation}
which implies the lower semicontinuity of $G$.
\end{proof}

\section{Subsequential convergence results}\label{sec:convergence}
We first analyze the generalized conditional gradient method with nonmonotone linesearch (Algorithm \ref{alg:GCGM}).
The stepsize at each iteration is determined by backtracking technique to satisfy average-type nonmonotone Armijo condition \citep{zhang2004nonmonotone}.

\begin{algorithm}[H]
\caption{Nonmonotone generalized conditional gradient method}
    \label{alg:GCGM}
    \begin{algorithmic}
    \STATE {\bfseries Input:} $x^0\in\dom g,~ F_0=F(x^0),~ 0<\beta<1,~ 0<\sigma<1,~ 0<p\le1$, and $k=0$.
    \REPEAT
    \STATE Compute $v^k\in\argmin_{x\in\mathbb{E}}\left\{\innerprod{\nabla f(x^k)}{x}+g(x)\right\}$ and $d^k=v^k-x^k$.
    \STATE Find the smallest $i\in\{0,1,2,\ldots\}$ s.t.
    \begin{equation}\label{eq:acceptance-criterion}
        F(x^{k,i})\le F_k-\sigma\beta^iG(x^k)
    \end{equation}
    where
    \begin{equation}
        x^{k,i}=x^k+\beta^id^k.
    \end{equation}
    \STATE Denote $i_k=i$ and choose $p_{k+1}\in[p,1]$.
    \STATE Set $x^{k+1}=x^{k,i_k},~ F_{k+1}=p_{k+1}F(x^{k+1})+(1-p_{k+1})F_k$, and $k\leftarrow k+1$.
    \UNTIL Terminated criterion is satisfied.
    \end{algorithmic}
\end{algorithm}

If $\sigma\le1/2$ and $p=1$, Algorithm \ref{alg:GCGM} reduces to that of \citet{kunisch2024fast}.
\citet{upadhayay2024nonmonotone} have proposed a conditional gradient method using average-type nonmonotone Armijo condition for multi-objective optimization problems.
Their algorithm for single-objective optimization coincides with Algorithm \ref{alg:GCGM} when $g$ is compact convex set and a specific choice of $p_k$ is made.

In the following, we provide a subsequential convergence result for Algorithm \ref{alg:GCGM} without smoothness assumptions.
Using the convexity of $g$, we have the following descent property.

\begin{lemma}\label{lem:descent-property}
Let $x\in\dom g$, $t\in[0,1]$, and $d=v^*-x$, where $v^*$ is a solution of \eqref{eq:subproblem}.
Then, it holds that
\begin{equation}
    F(x+td)\le F(x)-tG(x)+o(t).
\end{equation}
\end{lemma}

\begin{proof}
We see from the convexity of $g$ that
\begin{align}
    F(x+td)-F(x) &=f(x+td)-f(x)+g(x+td)-g(x)\\
    &\le f(x+td)-f(x)+tg(x+d)+(1-t)g(x)-g(x)\\
    &=\innerprod{\nabla f(x)}{td}+o(t)+t(g(x+d)-g(x))\\
    &=t(\innerprod{\nabla f(x)}{v^*-x}+g(v^*)-g(x))+o(t)\\
    &=-tG(x)+o(t),
\end{align}
which is the desired result.
\end{proof}

Using Lemma \ref{lem:descent-property}, the well definedness of Algorithm \ref{alg:GCGM} is established as follows.

\begin{lemma}\label{lem:well-definedness}
The linesearch procedure in Algorithm \ref{alg:GCGM} terminates in a finite number of iterations.
\end{lemma}

\begin{proof}
From Lemma \ref{lem:descent-property}, we have
\begin{equation}
    F(x^{k,i})=F(x^k+\beta^id^k)\le F(x^k)-\sigma\beta^iG(x^k)
\end{equation}
for all sufficiently large $i$.
Since $F_0=F(x^0)$ and it follows from the acceptance criterion \eqref{eq:acceptance-criterion} at the previous iteration that
\begin{equation}\label{eq:lower-bound}
    F_k=p_kF(x^k)+(1-p_k)F_{k-1}\ge p_kF(x^k)+(1-p_k)\{F(x^k)+\sigma\beta^{i_{k-1}}G(x^{k-1})\}\ge F(x^k)
\end{equation}
for $k\ge1$, the acceptance criterion \eqref{eq:acceptance-criterion} holds for all sufficiently large $i$.
\end{proof}

The following properties hold for Algorithm \ref{alg:GCGM}.

\begin{proposition}\label{prop:property-GCGM}
Let $\{x^k\}$ be a sequence generated by Algorithm \ref{alg:GCGM}.
Then the following assertions hold:
\begin{enumerate}[(i)]
    \item The sequence $\{F_k\}$ is monotonically nonincreasing and bounded from below by $\inf_{x\in\mathbb{E}}F(x)$.
    In particular, it holds that
    \begin{equation}
        F(x^{k+1})\le F_{k+1}\le F_k-p\sigma\beta^{i_k}G(x^k)
    \end{equation}
    for all $k\ge0$;
    \item The sequences $\{F_k\}$ and $\{F(x^k)\}$ converge to a same finite value;
    \item The sequence $\{x^k\}$ is included in the lower level set $\{x\in\mathbb{E}\mid F(x)\le F(x^0)\}\subset\dom g$;
    \item Let $\{x^k\}_K$ be a subsequence of $\{x^k\}$ converging to some point $x^*$. Then, $\{v^k\}_K$ is bounded.
\end{enumerate}
\end{proposition}

\begin{proof}
Using the acceptance criterion \eqref{eq:acceptance-criterion}, we have
\begin{equation}\label{eq:monotonicity-GCGM}
    F_{k+1}\le p_{k+1}\{F_k-\sigma\beta^{i_k}G(x^k)\}+(1-p_{k+1})F_k\le F_k-p\sigma\beta^{i_k}G(x^k).
\end{equation}
The lower bound is obtained as in \eqref{eq:lower-bound}.
Since the sequence $\{F_k\}$ is monotonically nonincreasing and bounded from below, $\{F_k\}$ converges to a finite value.
On the other hand, it follows from the definition and monotonicity of $F_k$ that
\begin{equation}
    F_k\ge F(x^k)=F_{k-1}+\frac{F_k-F_{k-1}}{p_k}\ge F_{k-1}+\frac{F_k-F_{k-1}}{p},
\end{equation}
which implies that $\{F(x^k)\}$ converges to the same limit as $\{F_k\}$.
As $\{F_k\}$ is monotonically nonincreasing, one has $F(x^k)\le F_k\le F_0=F(x^0)$.
Let $\{x^k\}_K$ be a subsequence of $\{x^k\}$ converging to some point $x^*$.
We see from the continuity of $\nabla f$ that $\{\nabla f(x^k)\}_K$ is bounded.
To derive a contradiction, suppose that $\{v^k\}_K$ is unbounded.
Without loss of generality, we may assume that $\lim_{k\to\infty, k\in K}\|v_k\|=\infty$.
From the optimality of the subproblem, we have
\begin{equation}
    g(v^k)\le\innerprod{\nabla f(x^k)}{x^*-v^k}+g(x^*)\le\|\nabla f(x^k)\|\|x^*-v^k\|+g(x^*).
\end{equation}
Taking the partial limit yields the contradiction by the supercoerciveness of $g$, and hence $\{v^k\}_K$ is bounded.
\end{proof}

A subsequential convergence result for Algorithm \ref{alg:GCGM} is described as follows.

\begin{theorem}\label{thm:subsequential-convergence}
Any accumulation point of $\{x^k\}$ generated by Algorithm \ref{alg:GCGM} is a stationary point of \eqref{problem:general}.
\end{theorem}

\begin{proof}
Let $\{x^k\}_K$ be a subsequence of $\{x^k\}$ converging to some point $x^*$.
If $\liminf_{k\to\infty, k\in K}\beta^{i_k}>0$, it follows from Proposition \ref{prop:property-GCGM} and Lemma \ref{lem:lsc-F-W-gap} that
\begin{align}
    0=\liminf_{k\to\infty, k\in K}F_k-F_{k+1} &\ge\liminf_{k\to\infty, k\in K}p\sigma\beta^{i_k}G(x^k)\\
    &\ge p\sigma\left(\liminf_{k\to\infty, k\in K}\beta^{i_k}\right)\liminf_{k\to\infty, k\in K}G(x^k)\\
    &\ge p\sigma\left(\liminf_{k\to\infty, k\in K}\beta^{i_k}\right)G(x^*).
\end{align}
Thus, we have $G(x^*)=0$, which indicates the stationarity of $x^*$ by Lemma \ref{lem:stationarity-F-W-gap}.

The remaining part of this proof assumes that $\liminf_{k\to\infty, k\in K}\beta^{i_k}=0$.
Without loss of generality, we may suppose that $\beta^{i_k}\to_K0$ and that $i_k\ge1$ holds for all $k\in K$, which implies that
\begin{equation}\label{eq:violation}
    F(\hat{x}^k)>F_k-\sigma\beta^{i_k-1}G(x^k)\ge F(x^k)-\sigma\beta^{i_k-1}G(x^k)
\end{equation}
where $\hat{x}^k\coloneqq x^{k,i_k-1}$ and the last inequality follows from Proposition \ref{prop:property-GCGM} (i).

From Proposition \ref{prop:property-GCGM} (iv), $\{v^k\}_K$ is bounded, and hence $\{d^k\}_K$ is also bounded.
Since $\hat{x}^k-x^k=\beta^{i_k-1}d^k$, it holds that $\hat{x}^k\to_Kx^*$.
On the other hand, by the mean-value theorem, there exist a convex combination $\xi^k$ of $x^k$ and $\hat{x}^k$ such that
\begin{equation}
    F(x^k)-F(\hat{x}^k)=g(x^k)-g(\hat{x}^k)-\innerprod{\nabla f(\xi^k)}{\hat{x}^k-x^k}.
\end{equation}
Combining this with \eqref{eq:violation} yields
\begin{align}
    \sigma\beta^{i_k-1}G(x^k) &>F(x^k)-F(\hat{x}^k)\\
    &=g(x^k)-g(\hat{x}^k)-\innerprod{\nabla f(\xi^k)}{\hat{x}^k-x^k}\\
    &\ge\beta^{i_k-1}(g(x^k)-g(v^k))-\innerprod{\nabla f(\xi^k)}{\hat{x}^k-x^k}\\
    &=\beta^{i_k-1}(g(x^k)-g(v^k))-\innerprod{\nabla f(x^k)}{\beta^{i_k-1}d^k}+\innerprod{\nabla f(x^k)-\nabla f(\xi^k)}{\beta^{i_k-1}d^k}\\
    &=\beta^{i_k-1}G(x^k)+\beta^{i_k-1}\innerprod{\nabla f(x^k)-\nabla f(\xi^k)}{d^k}.
\end{align}
By rearranging the above and dividing both sides by $\beta^{i_k-1}$, we obtain
\begin{equation}\label{eq:bound-F-W-gap}
    (1-\sigma)G(x^k)\le\innerprod{\nabla f(\xi^k)-\nabla f(x^k)}{d^k}\le\|\nabla f(\xi^k)-\nabla f(x^k)\|\|d^k\|.
\end{equation}
Since both $\{x^k\}_K$ and $\{\hat{x}^k\}_K$ converge to $x^*$, $\{\xi^k\}_K$ also converges to $x^*$.
Using the continuity of $\nabla f$, the boundedness of $\{d^k\}_K$, and Lemma \ref{lem:lsc-F-W-gap}, we see from \eqref{eq:bound-F-W-gap} that
\begin{equation}
    G(x^*)\le\liminf_{k\to\infty, k\in K}G(x^k)\le\lim_{k\to\infty, k\in K}\frac{1}{1-\sigma}\|\nabla f(\xi^k)-\nabla f(x^k)\|\|d^k\|=0.
\end{equation}
This completes the proof.
\end{proof}

Theorem \ref{thm:subsequential-convergence} is an extension of that of \citet[Section 2.2]{bertsekas1999nonlinear}, which is the result for the case where $p=1$ and $g$ is compact convex set.
By carefully examining the proof of the previous theorem, we see that if a convergent subsequence $\{x^k\}_K$ exists, then it holds that $\liminf_{k\to\infty, k\in K}G(x^k)=0$.
From this, we obtain the following result.

\begin{corollary}\label{cor:F-G-gap-convergence}
If there exists an accumulation point of $\{x^k\}$ generated by Algorithm \ref{alg:GCGM}, then we have
\begin{equation}
    \lim_{k\to\infty}\min_{0\le l\le k}G(x^l)=0.
\end{equation}
\end{corollary}

From Corollary \ref{cor:F-G-gap-convergence}, we see that if an accumulation point exists, for any $\varepsilon>0$ we can obtain a point $x$ satisfying
\begin{equation}
    G(x)\le\varepsilon
\end{equation}
within a finite number of iterations.

\subsection{Convergence result for parameter-free variant}
Next, we conduct a convergence analysis of the parameter-free generalized conditional gradient method proposed by \citet{ito2023parameter} (Algorithm \ref{alg:p-f-GCGM}) without smoothness assumption.
The parameter-free variant employs a more complicated backtracking procedure than Algorithm \ref{alg:GCGM}.

\begin{algorithm}[H]
\caption{Parameter-free generalized conditional gradient method \citep{ito2023parameter}}
    \label{alg:p-f-GCGM}
    \begin{algorithmic}
    \STATE {\bfseries Input:} $x^0\in\dom g,~ L_{-1}>0$, and $k=0$.
    \REPEAT
    \STATE Compute $v^k\in\argmin_{x\in\mathbb{E}}\left\{\innerprod{\nabla f(x^k)}{x}+g(x)\right\}$ and $d^k=v^k-x^k$.

    \STATE Find the smallest $i\in\{0,1,2,\ldots\}$ s.t.
    \begin{equation}\label{eq:acceptance-criterion-p-f}
        F(x^{k,i})\le F(x^k)-\frac{1}{2}\tau_k^{(i)}G(x^k)+\frac{1}{2}L_k^{(i)}(\tau_k^{(i)})^2\|d^k\|^2
    \end{equation}
    where
    \begin{align}
        L_k^{(i)} &=2^{i-1}L_{k-1},\\
        \tau_k^{(i)} &=\min\{1,G(x^k)/(2L_k^{(i)}\|d^k\|^2)\},\\
        x^{k,i} &=x^k+\tau_k^{(i)}d^k.
    \end{align}
    \STATE Set $x^{k+1}=x^{k,i},~ L_k=L_k^{(i)},~ \tau_k=\tau_k^{(i)}$, and $k\leftarrow k+1$.
    \UNTIL Terminated criterion is satisfied.
    \end{algorithmic}
\end{algorithm}

\citet{ito2023parameter} have showed that under the $\nu$-H\"older continuity of $\nabla f$ and the boundedness of $\dom g$, the iteration complexity of Algorithm \ref{alg:p-f-GCGM} for reaching $G(x^k)\le\varepsilon$ is $O(\varepsilon^{-1-1/\nu})$.
When the problem is assumed to have the uniform convex structure with modulus $\rho\ge2$\footnote{We say that the problem \eqref{problem:general} has a uniform convex structure if there exist $\kappa>0$ and $\rho\ge2$ such that for any $x\in\dom g$ and any solution $v^*$ of \eqref{eq:subproblem}, it holds that $\innerprod{\nabla f(x)}{v}+g(v)-\innerprod{\nabla f(x)}{v^*}-g(v^*)\ge\frac{\kappa}{\rho}\|v-v^*\|^\rho$ for all $v\in\mathbb{E}$.} instead of the boundedness of $\dom g$, the iteration complexity has been shown to be $O(\varepsilon^{-1-(\rho-1-\nu)/(\rho\nu)})$.
Furthermore, if the convexity of $f$ is assumed, the above complexity can be improved.
Many instances of problem \eqref{problem:general} satisfy the H\"older continuity assumption of $\nabla f$.
In fact, when $\dom g$ is compact and $\nabla f:\mathbb{E}\to\mathbb{E}$ is semi-algebraic, this assumption holds \citep[Proposition C.1]{bolte2020ah}.
However, determining in advance whether $\nabla f$ is semi-algebraic or, more generally, whether $\nabla f$ is H\"older continuous is not always easy and may not be user-friendly.

Here, we show a convergence result of Algorithm \ref{alg:p-f-GCGM} without the H\"older continuity of the gradient.
This ensures that Algorithm \ref{alg:p-f-GCGM} can be applied safely even when it is unclear whether the H\"older condition is satisfied.
We first show the well definedness of the algorithm.

\begin{lemma}\label{lem:well-definedness-p-f}
The linesearch procedure in Algorithm \ref{alg:p-f-GCGM} terminates in a finite number of iterations.
\end{lemma}

\begin{proof}
From Lemma \ref{lem:descent-property}, we have
\begin{equation}
    F(x^{k,i})=F(x^k+\tau_k^{(i)}d^k)\le F(x^k)-\frac{1}{2}\tau_k^{(i)}G(x^k)
\end{equation}
for all sufficiently large $i$.
Thus, the acceptance criterion \eqref{eq:acceptance-criterion-p-f} holds for all sufficiently large $i$.
\end{proof}

The following properties hold for Algorithm \ref{alg:p-f-GCGM}.

\begin{proposition}\label{prop:property-p-f-GCGM}
Let $\{x^k\}$ be a sequence generated by Algorithm \ref{alg:p-f-GCGM}.
Then the following assertions hold:
\begin{enumerate}[(i)]
    \item The sequence $\{F(x^k)\}$ is monotonically nonincreasing and bounded from below by $\inf_{x\in\mathbb{E}}F(x)$.
    In particular, it holds that
    \begin{equation}
        F(x^{k+1})\le F(x^k)-\frac{1}{4}\tau_kG(x^k)
    \end{equation}
    for all $k\ge0$;
    \item The sequence $\{F(x^k)\}$ converges to a finite value;
    \item The sequence $\{x^k\}$ is included in the lower level set $\{x\in\mathbb{E}\mid F(x)\le F(x^0)\}\subset\dom g$;
    \item Let $\{x^k\}_K$ be a subsequence of $\{x^k\}$ converging to some point $x^*$. Then, $\{v^k\}_K$ is bounded.
\end{enumerate}
\end{proposition}

\begin{proof}
From the definition of $\tau_k$, it holds that
\begin{equation}
    2L_k\tau_k\|d^k\|^2\le G(x^k).
\end{equation}
Using the acceptance criterion \eqref{eq:acceptance-criterion-p-f}, we have
\begin{align}\label{eq:monotonicity-p-f-GCGM}
    F(x^{k+1}) &\le F(x^k)-\frac{1}{2}\tau_kG(x^k)+\frac{1}{2}L_k(\tau_k)^2\|d^k\|^2\\
    &\le F(x^k)-\frac{1}{2}\tau_kG(x^k)+\frac{1}{4}\tau_kG(x^k)\\
    &=F(x^k)-\frac{1}{4}\tau_kG(x^k).
\end{align}
Since the sequence $\{F(x^k)\}$ is monotonically nonincreasing and bounded from below, $\{F(x^k)\}$ converges to a finite value.
As $\{F(x^k)\}$ is monotonically nonincreasing, one has $F(x^k)\le F(x^0)$.
Let $\{x^k\}_K$ be a subsequence of $\{x^k\}$ converging to some point $x^*$.
We see from the continuity of $\nabla f$ that $\{\nabla f(x^k)\}_K$ is bounded.
To derive a contradiction, suppose that $\{v^k\}_K$ is unbounded.
Without loss of generality, we may assume that $\lim_{k\to\infty, k\in K}\|v_k\|=\infty$.
From the optimality of the subproblem, we have
\begin{equation}
    g(v^k)\le\innerprod{\nabla f(x^k)}{x^*-v^k}+g(x^*)\le\|\nabla f(x^k)\|\|x^*-v^k\|+g(x^*).
\end{equation}
Taking the partial limit yields the contradiction by the supercoerciveness of $g$, and hence $\{v^k\}_K$ is bounded.
\end{proof}

Using Proposition \ref{prop:property-p-f-GCGM}, we obtain the following convergence result.

\begin{theorem}\label{thm:subsequential-convergence-p-f}
Suppose that the lower level set $\{x\in\mathbb{E}\mid F(x)\le F(x^0)\}$ is bounded.
Let $\{x^k\}$ be a sequence generated by Algorithm \ref{alg:p-f-GCGM}.
Then there exists an accumulation point of $\{x^k\}$ being a stationary point of \eqref{problem:general} and it holds that
\begin{equation}
    \lim_{k\to\infty}\min_{0\le l\le k}G(x^l)=0.
\end{equation}
\end{theorem}

\begin{proof}
Since the lower level set $\{x\in\mathbb{E}\mid F(x)\le F(x^0)\}$ is bounded, it follows from Proposition \ref{prop:property-p-f-GCGM} (iii) that $\{x^k\}$ is bounded.
We first consider the case where $\{L_k\}$ is bounded. 
There is a subsequence $\{x^k\}_K$ of $\{x^k\}$ converging to some point $x^*$.

If $\liminf_{k\to\infty, k\in K}\tau_k>0$, it follows from Proposition \ref{prop:property-p-f-GCGM} that
\begin{align}
    0=\liminf_{k\to\infty, k\in K}F(x^k)-F(x^{k+1}) &\ge\liminf_{k\to\infty, k\in K}\frac{1}{4}\tau_kG(x^k)\\
    &\ge \frac{1}{4}\left(\liminf_{k\to\infty, k\in K}\tau_k\right)\liminf_{k\to\infty, k\in K}G(x^k),
\end{align}
which implies that $\lim_{k\to\infty}\min_{0\le l\le k}G(x^l)=0$.
On the other hand, as we see from Lemma \ref{lem:lsc-F-W-gap} that $G(x^*)=0$, $x^*$ is a stationary point by Lemma \ref{lem:stationarity-F-W-gap}.

If $\liminf_{k\to\infty, k\in K}\tau_k=0$, there is a subsequence $\{\tau_k\}_{\tilde{K}}$ of $\{\tau_k\}_K$ such that $\tau_k\to_{\tilde{K}}0$ and that $\tau_k<1$ holds for all $k\in\tilde{K}$.
Then, it holds that $\tau_k=G(x^k)/(2L_k\|d^k\|^2)$ for $k\in\tilde{K}$.
From Proposition \ref{prop:property-p-f-GCGM} (iv), $\{v^k\}_{\tilde{K}}$ is bounded, and hence $\{d^k\}_{\tilde{K}}$ is also bounded.
Since $\{L_k\}_{\tilde{K}}$ and $\{\|d^k\|\}_{\tilde{K}}$ are bounded, we have $G(x^k)\to_{\tilde{K}}0$.
Consequently, $\lim_{k\to\infty}\min_{0\le l\le k}G(x^l)=0$ and $x^*$ is a stationary point.

Next, we consider the case where $\{L_k\}$ is unbounded.
Let $k_i\coloneqq\inf\{k\ge0\mid L_k/L_{-1}\ge2^i\}<\infty$ for $i\ge1$ and $K_0\coloneqq\{k_1,k_2,\ldots\}$, then $K_0$ is an infinite set and $L_k\to_{K_0}\infty$.
For $k\in K_0$, as backtracking is performed, it holds that
\begin{equation}\label{eq:violation-p-f}
    F(\hat{x}^k)>F(x^k)-\frac{1}{2}\min\{1,2\tau_k\}G(x^k)+\frac{1}{4}L_k\min\{1,2\tau_k\}^2\|d^k\|^2\ge F(x^k)-\frac{1}{2}\min\{1,2\tau_k\}G(x^k),
\end{equation}
where $\hat{x}^k\coloneqq x^k+\min\{1,2\tau_k\}d^k$.
Since $\{x^k\}_{K_0}$ is bounded, there exists a subsequence $\{x^k\}_K$ of $\{x^k\}_{K_0}$ converging to some point $x^*$.

If $\liminf_{k\to\infty, k\in K}\tau_k>0$, by proceeding in the same manner as in the case where $\{L_K\}$ is bounded, we obtain the desired result.

The remaining part of this proof assumes that $\liminf_{k\to\infty, k\in K}\tau_k=0$.
There is a subsequence $\{\tau_k\}_{\tilde{K}}$ of $\{\tau_k\}_K$ such that $\tau_k\to_{\tilde{K}}0$ and that $\tau_k<1/2$ holds for all $k\in\tilde{K}$.
Since $\{d^k\}_{\tilde{K}}$ is bounded because of Proposition \ref{prop:property-p-f-GCGM} (iv), $\{\hat{x}^k\}_{\tilde{K}}$ also converges to $x^*$.
By the mean-value theorem, there exist a convex combination $\xi^k$ of $x^k$ and $\hat{x}^k$ such that
\begin{equation}
    F(x^k)-F(\hat{x}^k)=g(x^k)-g(\hat{x}^k)-\innerprod{\nabla f(\xi^k)}{\hat{x}^k-x^k}.
\end{equation}
Combining this with \eqref{eq:violation-p-f} yields
\begin{align}
    \tau_kG(x^k) &>F(x^k)-F(\hat{x}^k)\\
    &=g(x^k)-g(\hat{x}^k)-\innerprod{\nabla f(\xi^k)}{\hat{x}^k-x^k}\\
    &\ge2\tau_k(g(x^k)-g(v^k))-\innerprod{\nabla f(\xi^k)}{\hat{x}^k-x^k}\\
    &=2\tau_k(g(x^k)-g(v^k))-\innerprod{\nabla f(x^k)}{2\tau_kd^k}+\innerprod{\nabla f(x^k)-\nabla f(\xi^k)}{2\tau_kd^k}\\
    &=2\tau_kG(x^k)+2\tau_k\innerprod{\nabla f(x^k)-\nabla f(\xi^k)}{d^k}
\end{align}
for $k\in\tilde{K}$.
By rearranging the above and dividing both sides by $\tau_k$, we obtain
\begin{equation}\label{eq:bound-F-W-gap-p-f}
    G(x^k)\le2\innerprod{\nabla f(\xi^k)-\nabla f(x^k)}{d^k}\le2\|\nabla f(\xi^k)-\nabla f(x^k)\|\|d^k\|.
\end{equation}
Since both $\{x^k\}_{\tilde{K}}$ and $\{\hat{x}^k\}_{\tilde{K}}$ converge to $x^*$, $\{\xi^k\}_{\tilde{K}}$ also converges to $x^*$.
Using the continuity of $\nabla f$ and the boundedness of $\{d^k\}_{\tilde{K}}$, we see from \eqref{eq:bound-F-W-gap-p-f} that
\begin{equation}
    \liminf_{k\to\infty, k\in\tilde{K}}G(x^k)\le\lim_{k\to\infty, k\in\tilde{K}}2\|\nabla f(\xi^k)-\nabla f(x^k)\|\|d^k\|=0,
\end{equation}
which implies that $\lim_{k\to\infty}\min_{0\le l\le k}G(x^l)=0$ and $x^*$ is a stationary point.
This completes the proof.
\end{proof}

Unlike the result for Algorithm \ref{alg:GCGM}, the additional assumption of the boundedness of the lower level set $\{x\in\mathbb{E}\mid F(x)\le F(x^0)\}$ is made.
By virtue of the supercoerciveness of $g$, the assumption is satisfied if there exist $a,b\in\mathbb{R}$ such that $f(x)\ge a\|x\|+b$ holds for all $x\in\mathbb{E}$.
In particular, if $f$ is a convex function, it is bounded from below by an affine function, which ensures that the assumption holds.
It goes without saying that the assumption is also satisfied when $\dom g$ is bounded.

According to Theorem \ref{thm:subsequential-convergence-p-f}, even if the H\"older continuity of $\nabla f$ does not hold, it is possible to find a point $x$ satisfying $G(x)\le\varepsilon$ in a finite number of iterations.
Therefore, even if the H\"older continuity of $\nabla f$ is unknown, Algorithm \ref{alg:p-f-GCGM} can still find an approximate stationary point, and if $\nabla f$ is H\"older continuous, the convergence rate automatically adapts to its H\"older exponent.

\section{Conclusion}\label{sec:conclusion}
The paper has provided subsequential convergence result for two linesearch-based generalized conditional gradient methods without smoothness assumptions.
Particularly, we have analyzed the generalized conditional gradient method proposed by \citet{ito2023parameter}, which automatically adapts to the H\"older exponent.
The analysis demonstrates that their algorithm can be applied safely even when it is unclear whether the smooth part $f$ satisfies the H\"older condition.

Recently, \citet{bolte2024iterates} have investigated whether the whole sequence generated by the conditional gradient methods converges.
They have provided counterexamples to the convergence of sequence under several stepsize strategies, but the case using Armijo-type linesearch is not covered.
Identifying the conditions under which the sequence generated by the linesearch-based generalized conditional gradient method converges is an interesting future work.

\section*{Acknowledgments}
This work was supported partly by the JSPS Grant-in-Aid for Early-Career Scientists 25K21158.
The author would like to thank Masaru Ito for his valuable comments.

\bibliography{reference.bib}

\begin{thebibliography}{30}
\providecommand{\natexlab}[1]{#1}
\providecommand{\url}[1]{\texttt{#1}}
\expandafter\ifx\csname urlstyle\endcsname\relax
  \providecommand{\doi}[1]{doi: #1}\else
  \providecommand{\doi}{doi: \begingroup \urlstyle{rm}\Url}\fi

\bibitem[Beck(2017)]{beck2017first}
Amir Beck.
\newblock \emph{First-order methods in optimization}.
\newblock SIAM, 2017.

\bibitem[Bertsekas(1999)]{bertsekas1999nonlinear}
Dimitri~P Bertsekas.
\newblock \emph{Nonlinear Programming}.
\newblock Athena scientific, 2nd edition, 1999.

\bibitem[Bolte et~al.(2020)Bolte, Glaudin, Pauwels, and Serrurier]{bolte2020ah}
J{\'e}r{\^o}me Bolte, Lilian Glaudin, Edouard Pauwels, and Mathieu Serrurier.
\newblock Ah$\backslash$" olderian backtracking method for min-max and min-min problems.
\newblock \emph{arXiv preprint arXiv:2007.08810}, 2020.

\bibitem[Bolte et~al.(2024)Bolte, Combettes, and Pauwels]{bolte2024iterates}
J{\'e}r{\^o}me Bolte, Cyrille~W Combettes, and Edouard Pauwels.
\newblock The iterates of the frank--wolfe algorithm may not converge.
\newblock \emph{Mathematics of Operations Research}, 49\penalty0 (4):\penalty0 2565--2578, 2024.

\bibitem[Braun et~al.(2022)Braun, Carderera, Combettes, Hassani, Karbasi, Mokhtari, and Pokutta]{braun2022conditional}
G{\'a}bor Braun, Alejandro Carderera, Cyrille~W Combettes, Hamed Hassani, Amin Karbasi, Aryan Mokhtari, and Sebastian Pokutta.
\newblock Conditional gradient methods.
\newblock \emph{arXiv preprint arXiv:2211.14103}, 2022.

\bibitem[Bredies et~al.(2009)Bredies, Lorenz, and Maass]{bredies2009generalized}
Kristian Bredies, Dirk~A Lorenz, and Peter Maass.
\newblock A generalized conditional gradient method and its connection to an iterative shrinkage method.
\newblock \emph{Computational Optimization and applications}, 42:\penalty0 173--193, 2009.

\bibitem[Clarkson(2010)]{clarkson2010coresets}
Kenneth~L Clarkson.
\newblock Coresets, sparse greedy approximation, and the frank-wolfe algorithm.
\newblock \emph{ACM Transactions on Algorithms (TALG)}, 6\penalty0 (4):\penalty0 1--30, 2010.

\bibitem[Combettes and Pokutta(2021)]{combettes2021complexity}
Cyrille~W Combettes and Sebastian Pokutta.
\newblock Complexity of linear minimization and projection on some sets.
\newblock \emph{Operations Research Letters}, 49\penalty0 (4):\penalty0 565--571, 2021.

\bibitem[De~Marchi(2023)]{de2023proximal}
Alberto De~Marchi.
\newblock Proximal gradient methods beyond monotony.
\newblock \emph{Journal of Nonsmooth Analysis and Optimization}, 4\penalty0 (Original research articles), 2023.

\bibitem[De~Marchi and Themelis(2022)]{de2022proximal}
Alberto De~Marchi and Andreas Themelis.
\newblock Proximal gradient algorithms under local lipschitz gradient continuity: A convergence and robustness analysis of panoc.
\newblock \emph{Journal of Optimization Theory and Applications}, 194\penalty0 (3):\penalty0 771--794, 2022.

\bibitem[Dunn(1979)]{dunn1979rates}
Joseph~C Dunn.
\newblock Rates of convergence for conditional gradient algorithms near singular and nonsingular extremals.
\newblock \emph{SIAM Journal on Control and Optimization}, 17\penalty0 (2):\penalty0 187--211, 1979.

\bibitem[Dunn and Harshbarger(1978)]{dunn1978conditional}
Joseph~C Dunn and S~Harshbarger.
\newblock Conditional gradient algorithms with open loop step size rules.
\newblock \emph{Journal of Mathematical Analysis and Applications}, 62\penalty0 (2):\penalty0 432--444, 1978.

\bibitem[Frank and Wolfe(1956)]{frank1956algorithm}
Marguerite Frank and Philip Wolfe.
\newblock An algorithm for quadratic programming.
\newblock \emph{Naval research logistics quarterly}, 3\penalty0 (1-2):\penalty0 95--110, 1956.

\bibitem[Fukushima and Mine(1981)]{fukushima1981generalized}
Masao Fukushima and Hisashi Mine.
\newblock A generalized proximal point algorithm for certain non-convex minimization problems.
\newblock \emph{International Journal of Systems Science}, 12\penalty0 (8):\penalty0 989--1000, 1981.

\bibitem[Ghadimi(2019)]{ghadimi2019conditional}
Saeed Ghadimi.
\newblock Conditional gradient type methods for composite nonlinear and stochastic optimization.
\newblock \emph{Mathematical Programming}, 173:\penalty0 431--464, 2019.

\bibitem[Ito et~al.(2023)Ito, Lu, and He]{ito2023parameter}
Masaru Ito, Zhaosong Lu, and Chuan He.
\newblock A parameter-free conditional gradient method for composite minimization under h{\"o}lder condition.
\newblock \emph{Journal of Machine Learning Research}, 24\penalty0 (166):\penalty0 1--34, 2023.

\bibitem[Jaggi(2013)]{jaggi2013revisiting}
Martin Jaggi.
\newblock Revisiting frank-wolfe: Projection-free sparse convex optimization.
\newblock In \emph{International conference on machine learning}, pages 427--435. PMLR, 2013.

\bibitem[Jia et~al.(2023)Jia, Kanzow, and Mehlitz]{jia2023convergence}
Xiaoxi Jia, Christian Kanzow, and Patrick Mehlitz.
\newblock Convergence analysis of the proximal gradient method in the presence of the kurdyka--{\l}ojasiewicz property without global lipschitz assumptions.
\newblock \emph{SIAM Journal on Optimization}, 33\penalty0 (4):\penalty0 3038--3056, 2023.

\bibitem[Jiang et~al.(2019)Jiang, Lin, Ma, and Zhang]{jiang2019structured}
Bo~Jiang, Tianyi Lin, Shiqian Ma, and Shuzhong Zhang.
\newblock Structured nonconvex and nonsmooth optimization: algorithms and iteration complexity analysis.
\newblock \emph{Computational Optimization and Applications}, 72\penalty0 (1):\penalty0 115--157, 2019.

\bibitem[Kanzow and Lehmann(2024)]{kanzow2024convergence}
Christian Kanzow and Leo Lehmann.
\newblock Convergence of nonmonotone proximal gradient methods under the kurdyka-lojasiewicz property without a global lipschitz assumption.
\newblock \emph{arXiv preprint arXiv:2411.12376}, 2024.

\bibitem[Kanzow and Mehlitz(2022)]{kanzow2022convergence}
Christian Kanzow and Patrick Mehlitz.
\newblock Convergence properties of monotone and nonmonotone proximal gradient methods revisited.
\newblock \emph{Journal of Optimization Theory and Applications}, 195\penalty0 (2):\penalty0 624--646, 2022.

\bibitem[Kunisch and Walter(2024)]{kunisch2024fast}
Karl Kunisch and Daniel Walter.
\newblock On fast convergence rates for generalized conditional gradient methods with backtracking stepsize.
\newblock \emph{Numerical Algebra, Control and Optimization}, 14\penalty0 (1):\penalty0 108--136, 2024.

\bibitem[Lacoste-Julien(2016)]{lacoste2016convergence}
Simon Lacoste-Julien.
\newblock Convergence rate of frank-wolfe for non-convex objectives.
\newblock \emph{arXiv preprint arXiv:1607.00345}, 2016.

\bibitem[Levitin and Polyak(1966)]{levitin1966constrained}
Evgeny~S Levitin and Boris~T Polyak.
\newblock Constrained minimization methods.
\newblock \emph{USSR Computational mathematics and mathematical physics}, 6\penalty0 (5):\penalty0 1--50, 1966.

\bibitem[Mine and Fukushima(1981)]{mine1981minimization}
Hisashi Mine and Masao Fukushima.
\newblock A minimization method for the sum of a convex function and a continuously differentiable function.
\newblock \emph{Journal of Optimization Theory and Applications}, 33:\penalty0 9--23, 1981.

\bibitem[Nesterov(2018)]{nesterov2018complexity}
Yu~Nesterov.
\newblock Complexity bounds for primal-dual methods minimizing the model of objective function.
\newblock \emph{Mathematical Programming}, 171\penalty0 (1):\penalty0 311--330, 2018.

\bibitem[Takahashi et~al.(2025)Takahashi, Pokutta, and Takeda]{takahashi2025accelerated}
Shota Takahashi, Sebastian Pokutta, and Akiko Takeda.
\newblock Accelerated convergence of frank--wolfe algorithms with adaptive bregman step-size strategy.
\newblock \emph{arXiv preprint arXiv:2504.04330}, 2025.

\bibitem[Upadhayay et~al.(2024)Upadhayay, Ghosh, Jauny, Yao, and Zhao]{upadhayay2024nonmonotone}
Ashutosh Upadhayay, Debdas Ghosh, Jauny, Jen-Chih Yao, and Xiaopeng Zhao.
\newblock A nonmonotone conditional gradient method for multiobjective optimization problems.
\newblock \emph{Soft Computing}, pages 1--22, 2024.

\bibitem[Yagishita and Ito(2025)]{yagishita2025proximal}
Shotaro Yagishita and Masaru Ito.
\newblock Proximal gradient-type method with generalized distance and convergence analysis without global descent lemma.
\newblock \emph{arXiv preprint arXiv:2505.00381}, 2025.

\bibitem[Zhang and Hager(2004)]{zhang2004nonmonotone}
Hongchao Zhang and William~W Hager.
\newblock A nonmonotone line search technique and its application to unconstrained optimization.
\newblock \emph{SIAM journal on Optimization}, 14\penalty0 (4):\penalty0 1043--1056, 2004.

\end{thebibliography}
\bibliographystyle{plainnat}

\end{document}